\documentclass[10pt]{article}
\usepackage{amsmath,amsthm}
\usepackage{amsfonts,amssymb}
\usepackage{enumerate}
\usepackage{makecell}
\usepackage{empheq}

\usepackage{cite}
\usepackage{array}
\newcolumntype{M}{>{\centering\arraybackslash}m{\dimexpr.25\linewidth-2\tabcolsep}}
  
\newtheorem{theorem}{Theorem}[section]

\newtheorem{lemma}[theorem]{Lemma}
\newtheorem{proposition}[theorem]{Proposition}

\newtheorem{definition}[theorem]{Definition}
\newtheorem{remark}[theorem]{Remark}

\theoremstyle{remark}

\newcommand{\mR}{\mathbb{R}}
\newcommand{\mC}{\mathbb{C}}

\newcommand{\mE}{\mathbb{E}}

\newcommand{\mS}{\mathbb{S}}

\newcommand{\cF}{\mathcal{F}}

\newcommand{\cK}{\mathcal{K}}

\newcommand{\ux}{\underline{x}}

\newcommand{\uy}{\underline{y}}
\newcommand{\uu}{\underline{u}}

\newcommand{\uz}{\underline{z}}

\newcommand{\uo}{\underline{\omega}}

\newcommand{\ueta}{\underline{\eta}}
\newcommand{\uzeta}{\underline{\zeta}}
\newcommand{\unu}{\underline{\nu}}

\newcommand{\bfx}{\mathbf{x}}
\newcommand{\bfy}{\mathbf{y}}
\newcommand{\bfz}{\mathbf{z}}
\newcommand{\bfu}{\mathbf{u}}

\newcommand{\pR}{\partial_{r}}

\begin{document}
\title
{Slice Fourier transform and convolutions}
\author{Lander Cnudde, Hendrik De Bie}

\date{}
\maketitle
\begin{abstract}
Recently the construction of various integral transforms for slice monogenic functions has gained a lot of attention. In line with these developments, the article at hand introduces the slice Fourier transform. In the first part, the kernel function of this integral transform is constructed using the Mehler formula. An explicit expression for the integral transform is obtained and allows for the study of its properties. In the second part, two kinds of corresponding convolutions are examined: Mustard convolutions and convolutions based on generalised translation operators. The paper finishes by demonstrating the connection between both.\\\\
Keywords: Clifford-Hermite functions, slice Fourier transform, Mustard convolution, generalised convolution\\
MSC: 30G35, 33C45
\end{abstract}

\maketitle

\section{Introduction}

\indent In recent years a lot of work has been put into the study of various integral transforms in a slice Clifford analytic setting. In line with earlier research on slice monogenic functions (see e.g. \cite{colombo2009slice,colombo2011noncommutative,Cnudde}), recent developments comprise the generalisation of various well-known integral transforms to this particular setting. So far the cases of the Bergman-Sce transform \cite{bergman-sce}, the Cauchy transform \cite{colombo2011noncommutative} and the (dual) Radon transform \cite{radon} have been treated. However, an integral transform of major importance has not yet been studied. The paper at hand introduces the classical Fourier transform to the slice monogenic framework, giving rise to the slice Fourier transform.\\
\indent In various papers (see e.g. \cite{brackx2005clifford,brackx2006two,de2011class,de2011clifford,de2012clifford}) the Fourier transform has been generalised to the general hypercomplex setting. Depending on the parity of the dimension of the underlying real Clifford algebra, these efforts even lead to closed forms for the corrresponding kernel functions.
The study of these new Fourier transforms has also been carried out from an application point of view, examining the possibility to use them for image processing (see e.g. \cite{Sangwine, SangwineCorrelation, Denis200774, Saliency}) . \\
\indent Up to now it has not been possible to generalise the Fourier transform to the slice Clifford analytic setting because a solid basis was missing. Owing to recent work on Clifford-Hermite functions and their properties (see \cite{Cnudde}), however, the insight in the corresponding function space has been growing. Hence it is now a natural question to address the construction of a slice Fourier transform based on these findings. In order to retrieve a realisation of the $\mathfrak{osp}(1|2)$ Lie superalgebra, \cite{Cnudde} manipulates the slice Cauchy-Riemann operator of \cite{colombo2013nonconstant} such that a slice Dirac operator is obtained. This in turn allowed for the construction of Clifford-Hermite functions which exhibit interesting differential properties. For this reason they qualify extremely well for an eigenfunction basis of the slice Fourier transform to be obtained.
\indent In a first part of this paper, the slice Fourier transform is constructed and its main properties are studied. The classical Fourier transform being an integral transform, the search for a generalisation to the above mentioned slice monogenic setting comes down to finding an appropriate kernel function. Next, the explicit expression of the integral transform allows for a closer study of its basic properties. It is also proved that the slice Fourier transform is well-defined on a subspace of a Clifford analogue of the Hilbert space $L^2$.\\
In this paper the construction of the kernel is achieved using the Mehler formula. Based on a set of eigenfunctions that are orthogonal with respect to a well-chosen inner product and their corresponding eigenvalues, this formula allows for a formal expression of the kernel function. 
Because the aforementioned Clifford-Hermite functions show the desired properties, they are declared eigenfunctions of the slice Fourier transform to be defined and assigned corresponding eigenvalues based on the scalar differential equation. Using their explicit expressions and several technical manipulations, we ultimately find a closed form for the integral transform.\\
Though this method merely starts from eigenfunctions and eigenvalues analogous to those of the classical Fourier transform, the resulting slice Fourier transform shows even more similarities. Not only does the final kernel expression comprise of familiar exponential functions, it is also proven that this kernel function obeys a system of Clifford algebra-valued partial differential equations analogous to the classical transform. At this point, a natural question to ask is whether it also yields a nice convolution property.\\
In the second part of this article, two interpretations of the classical convolution are explored to generalise the classical convolution property to the case of the slice Fourier transform. Based on \cite{Convolutions} both a Mustard convolution and a generalised translation operator are defined and proven to show the desired property. Given the non-commutative nature of the Clifford setting, the commutativity of the classical convolution can of course not be retrieved. One thus has to distinguish between left and right Mustard convolutions and generalised translation operators.\\
In the first approach, the (left) Mustard convolution $\star_S$ is defined such that its slice Fourier transform $\cF_S$ equals the product of the slice Fourier transforms of the constituting functions, or, for certain Clifford-valued functions $f$ and $g$,
\begin{align*}
\cF_S(f\star_S g)= \cF_S(f) \cF_S(g).
\end{align*}
In the second approach, the focus is on the translation operator $t_x$ in the definition of the classical convolution
\begin{align}\label{convolutie}
(f\star g) (x) = \frac{1}{\sqrt{2\pi}} \int_{\mR} t_y (f)(x) g(y) \mathrm{d}y,
\end{align}
where $t_y(f)(x)= f(x-y)$. The operator $t_y$ is such that its classical Fourier transform generates an extra factor identical to the classical kernel function, thus giving rise to the above convolution property. Based on these observations, a (left) generalised translation operator $T$ is defined which behaves analogously under the slice Fourier transform and leads to a slice convolution property as well.\\
Given that both approaches give rise to the same behaviour in the slice Fourier domain, we end by pinpointing the connection between them. \\
The paper is organised as follows. Section 2 summarises, besides some general preliminary results with respect to the slice Clifford context, the main results concerning the Clifford-Hermite functions studied in \cite{Cnudde}. In section 3, the explicit form of the slice Fourier transform is obtained using the Mehler formula and several of its properties are studied. Section 4 addresses two ways to define an appropriate convolution with respect to this transform. Apart from a Mustard-type convolution, the problem is also approached using generalised translation operators.

\section{Preliminaries}
This preliminary section contains some general background on the slice approach in $\mR^{m+1}$ and briefly summarizes the definitions and properties of the Clifford-Hermite functions as defined and proved in \cite{Cnudde}.

\subsection{Slice approach in $\mR^{m+1}$}
The $m+1$-dimensional real Clifford algebra $Cl_{m+1}$ has $m+1$ basis vectors $e_i, i=0,\ldots, m$, which satisfy the anti-commutation relations
$$e_ie_j+e_je_i=-2\delta_{ij}, \qquad i,j=0,\ldots,m.$$
A $k$-vector (where $k\leq m+1$) is an element $e_A$ of $Cl_{m+1}$ such that
$$e_A=e_{i_1}\ldots e_{i_k} $$
where $i_j\in \{0,\ldots,m \}$ for all $j\in \{1,\ldots,k\}$ and with $i_1<\ldots<i_k$. The variable $\bfx \in Cl_{m+1}$ is defined as the $1$-vector which corresponds to the $(m+1)$-tuple $(x_0,\ldots,x_m) \in \mR^{m+1}$ by
$$\bfx=x_0e_0+x_1e_1+\ldots+x_me_m.$$
Using spherical coordinates to describe the $Cl_m$-part $\ux$ of $\bfx$, one can also write
\begin{align*}
\bfx&=x_0e_0 + \ux\\
&=x_0e_0+r\uo,
\end{align*}
where $r=\sqrt{x_1^2+\ldots+x_m^2}$ and $\uo=\ux/r$. A general element $\bfx$ is thus defined by the triplet $(x_0,r,\uo)\in\mR\times \mR^+\times\mathbb{S}^{m-1}$, where $\mathbb{S}^{m-1}$ denotes the $(m-1)$-dimensional sphere in $\mR^m$. The variable $\bfx$ therefore lives in the subspace spanned by the fixed basis vector $e_0$ and the unit $1$-vector $\uo$.
This subspace is called a slice (see e.g.\cite{colombo2009slice} and the book \cite{colombo2011noncommutative}).\\
As a consequence, a general function $f$ of $\bfx$ will depend on $x_0, r$ and $\uo$. Throughout this article such functions $f$ will be written both as $f(\bfx)$ and as $f(x_0,r,\uo)$ because the former is more compact and the latter shows its dependences explicitly. Based on considerations in \cite{colombo2013nonconstant}, in \cite{Cnudde} the following definition was proposed.
\begin{definition}[Slice Dirac operator]
The slice Dirac operator $D_0$ is the partial differential operator defined as
$$D_0=e_0\partial_{x_0}+\uo\partial_r.$$
\end{definition}
\noindent Null-solutions of $D_0$ correspond to slice monogenic functions as studied in e.g.\cite{colombo2009slice,Gentili2007279,alternative}. Together with the multiplication operator $\bfx$ and the slice Dirac operator $D_0$, the Euler operator $\mE=\sum_{i=0}^m x_i\partial_{x_i}$ establishes a realisation of the $\mathfrak{osp}(1|2)$-superalgebra, see \cite{3032075}.

\begin{theorem}\label{osp} 
The operators $\mathbf{x}$, $D_0$ and $\mE$ constitute a Lie superalgebra, isomorphic with $\mathfrak{osp}(1|2)$, with relations
\begin{center}
\begin{tabular}{rlrl}
\emph{(i)} & $\{\bfx, \bfx \} = -2 \lvert \bfx \lvert^2$ & \emph{(ii)} & $ \{D_0, D_0 \} = -2 (\partial_{x_0}^2 + \pR^2)$\\
\emph{(iii)} & $\{\bfx, D_0 \} = -2 \left(\mE + 1\right)$ & \emph{(iv)} & $ [\mE + 1 , D_0 ] = -D_0$\\
\emph{(v)} & $[\lvert \bfx\lvert^2, D_0] = -2 \bfx$ & \emph{(vi)} & $ [\mE + 1, \bfx ] = \bfx$\\
\emph{(vii)} & $[\partial_{x_0}^2 + \pR^2, \bfx] = 2 D_0$ & \emph{(viii)} & $ [\mE + 1 , \partial_{x_0}^2 + \pR^2] = -2(\partial_{x_0}^2 + \pR^2)$\\
\emph{(ix)} & $[\partial_{x_0}^2 + \pR^2, \lvert \bfx\lvert^2] = 4 \left(\mE + 1\right)$ & \emph{(x)} & $ [\mE + 1 , \lvert \bfx\lvert^2 ] = 2\lvert \bfx \lvert^2$.
\end{tabular}
\end{center}
\end{theorem}
\begin{definition}The Clifford conjugation $\overline{\phantom{ll}}$ is defined as
\begin{align*}
\overline{\lambda} &= \lambda^* && \lambda \in \mC\\
\overline{e_i}& = -e_i && i=0,\ldots,m.\\
\overline{ab}&=\overline{b}\ \overline{a} && a,b\in Cl_{m+1}.
\end{align*}
where $^*$ denotes the standard complex conjugation.
\end{definition}
\begin{definition}The Hilbert space $\mathcal{L}^2$ is defined as
\begin{align*}
\mathcal{L}^2 &= L^{2}(\mR^{m+1},r^{1-m}\mathrm{d}\bfx)\ \otimes\ Cl_{m+1}\\
&=\left\{ f : \mR^{m+1} \rightarrow Cl_{m+1}\ \vline\ \left[ \int_{\mR^{m+1}} \overline{f(\bfx)} f(\bfx)\ r^{1-m}\ \mathrm{d}\bfx \right]_0 < +\infty  \right\}
\end{align*}
where $[\ .\ ]_0$ denotes the scalar part of the expression between the brackets. 
\end{definition}

\noindent On $\mR^{m+1}$ an inner product was defined to be self-adjoint with respect to $D_0$:
\begin{definition}
On the Hilbert module $\mathcal{L}^2$, regarded as a complex vector space, the inner product of two functions $f,g:\mR^{m+1}\rightarrow Cl_{m+1}$ is defined as
\begin{align*}
\langle f,g \rangle = \left[ \int_{\mR^{m+1}} \overline{f(\bfx)}g(\bfx)\ r^{1-m}\ \mathrm{d}\bfx \right]_0=\left[ \int_{\mR^{m+1}} \overline{f(\bfx)}g(\bfx)\ \mathrm{d}x_0 \mathrm{d}r \mathrm{d}\sigma_{\ux} \right]_0
\end{align*}
where $\mathrm{d}\sigma_{\ux}$ denotes the measure on the unit sphere $\mS^{m-1}$ corresponding to the $\ux$-part of $\bfx$.
\end{definition}
\noindent The inner product has some interesting properties.

\begin{proposition}
The inner product $\langle f,g \rangle = \int_{\mR^{m+1}}\overline{f}g\ \mathrm{d}x_0 \mathrm{d}r \mathrm{d} \uo$ on the right $Cl_{m+1}$-module $\mathcal{L}^2$ obeys the relations
\begin{align*}
\langle D_0 f, g\rangle &= \langle f, D_0 g\rangle,\\
\langle \bfx f, g\rangle &= - \langle f, \bfx g\rangle
\end{align*} on a dense subset of $\mathcal{L}^2$.
\end{proposition}

\subsection{The Clifford-Hermite functions}
Based on the classical definitions, the Clifford-Hermite polynomials and functions are defined using the kernel of the differential operator $D_0$ and the $\mathfrak{osp}(1|2)$-relations. In \cite{Cnudde} it was shown that the kernel of $D_0$ is a right $Cl_{m+1}$-module which is spanned by the homogeneous polynomials $m_k(\bfx)=\left(e_0-1\right) (x_0+\ux)^k$ of degree $k \in \mathbb{N}$.
\begin{definition}[Clifford-Hermite polynomials]
The Clifford-Hermite polynomials $h_{j,k}$ of degree $j$ and order $k$ are defined as
$$h_{j,k}(\bfx)m_k(\bfx)=(\bfx-cD_0)^j m_k(\bfx)$$
where $c\in \mR^+_0$ and $j\in\mathbb{N}$.
\end{definition}
\noindent The parameter $c$ adds some freedom to the definition. In order not to overload notation, however, its presence will not be denoted explicitly. 
We summarise their most important properties (see \cite{Cnudde}).
\begin{theorem}\label{vglC}
The polynomials $H_j(m_k)(\bfx)=h_{j,k}(\bfx)m_k(\bfx)$ are solutions of the differential equation
\begin{align*}
cD_0^2 H_{j}(m_k)(\bfx) -\bfx D_0 H_{j}(m_k)(\bfx)+ C(j,k) H_{j}(m_k)(\bfx) =0
\end{align*}
with $C(j,k)=-2t$ if $j=2t$ and $C(j,k)=-2(k+t+1)$ if $j=2t+1$.
\end{theorem}

\begin{theorem}\label{Laguerre}
The Hermite polynomials $h_{j,k}$ can be expressed as
\begin{align*}
h_{2t,k}(\bfx)&=(2c)^t t!\ L^k_t\left(\frac{\vert\bfx\vert^2}{2c}\right)\\
h_{2t+1,k}(\bfx)&=(2c)^t t!\ \bfx\ L^{k+1}_t\left(\frac{\vert\bfx\vert^2}{2c}\right)
\end{align*}
where $L_t^k$ are the generalised Laguerre polynomials of degree $t$ on the real line.
\end{theorem}

\noindent Based on the Clifford-Hermite polynomials, the Clifford-Hermite functions can be defined.
\begin{definition}[Clifford-Hermite functions]
The Clifford-Hermite functions $\psi_{j,k}$ are defined as
$$\psi_{j,k}(\bfx)=h_{j,k}(\bfx)\exp\left(-|\bfx|^2/4c\right)$$
where $c\in\mR_0^+$.
\end{definition}

\begin{proposition}
The Clifford-Hermite functions $\psi_{j,k}$ satisfy the relations
\begin{align}\label{psis}
\begin{cases}
\phantom{\tilde{D_c}^\dagger}\psi_{j,k}&=\phantom{tC(j,k)}\tilde{D_c}\psi_{j-1,k}\\
\tilde{D_c}^\dagger\psi_{j,k}&= -c\ C(j,k) \psi_{j-1,k}
\end{cases}
\end{align}
with $\tilde{D_c}=\frac{\bfx}{2} -cD_0$, $C(j,k)$ as in Theorem \ref{vglC} and where $^\dagger$ denotes the adjoint with respect to the inner product.
\end{proposition}

\begin{theorem}\label{ortho}
Let $\psi_{j_i,k_i}=(\bfx-cD_0)^{j_i}m_{k_i}(\bfx)\exp(-\vert \bfx \vert^2/4c)$ with $m_{k_i}(\bfx)=(e_0-1)(x_0+\ux)^{k_i}$ for $i=1,2$.
The inner product of these two Clifford-Hermite functions $\psi_{j_1,k_1}$ and $\psi_{j_2,k_2}$ is given by
\begin{align*}
\left\langle \psi_{j_1,k_1}, \psi_{j_2,k_2} \right\rangle = A(j_1,k_1) \delta_{j_1j_2} \delta_{k_1k_2}
\end{align*}
with
\begin{align*}A(j_1,k_1)=
\begin{cases}
2(2c)^{2t_1+k_1+1} t_1! (k_1+t_1)! \frac{\pi^{m/2+1}}{\Gamma(m/2)} & j_1=2t_1,\\
2(2c)^{2t_1+k_1+2} t_1! (k_1+t_1+1)! \frac{\pi^{m/2+1}}{\Gamma(m/2)} \qquad & j_1=2t_1+1.
\end{cases}
\end{align*}
\end{theorem}

\begin{theorem}\label{SDE}
The Clifford-Hermite functions $\psi_{j,k}$ are solutions of the scalar differential equation
\begin{align}\label{scalar}
\left(cD_0^2 + \frac{\vert\bfx\vert^2}{4c}\right) \psi_{j,k}(\bfx) =(j+k+1)\psi_{j,k}(\bfx).
\end{align}
\end{theorem}

\begin{definition}The set of finite linear combinations of Clifford-Hermite functions over $Cl_{m+1}$ will be denoted as $\mathcal{V}$, so $\mathcal{V}=\text{span}_{Cl}\{\psi_{j,k}\}$ and $\mathcal{V} \subset \mathcal{L}^2$.
\end{definition}
\begin{remark}\label{splits}
As a consequence any function $f\in\mathcal{V}$ can be written as $f(\bfx)=f_1(x_0,r)+\uo f_2(x_0,r)$, where $f_1$ and $f_2$ are $Cl_{m+1}$-valued functions defined on $\mR \times \mR^+$.
\end{remark}

\section{Slice Fourier transform}
Based on Theorem \ref{SDE}, a formal definition of the slice Fourier transform (see \cite{oscillatorsemigroup}) on $\mathcal{V}$ is obtained by taking the exponential of the scalar differential equation \eqref{scalar}:
\begin{align}\label{formaleigen}
e^{-i\frac{\pi}{2}\left(cD_0^2 + \frac{\vert\bfx\vert^2}{4c}\right)} \psi_{j,k}(\bfx) =(-i)^{(j+k+1)}\psi_{j,k}(\bfx)
\end{align}
where the choice for the sign of $-i$ is arbitrary.
The aim of this section is however to construct an explicit integral expression corresponding to the above formal series expansion. To this end functions $f$ will be restricted to $f \in\mathcal{V}$.
Given its eigenvalues and corresponding eigenfunctions, an appropriate kernel function $\cK^M$ will thus be constructed such that
\begin{equation}\label{general}
\mathcal{F}_S(f)(\bfy) = \int_{\mR^{m+1}} \mathcal{K}^M(\bfx,\bfy) f(\bfx)\ r^{1-m }\mathrm{d}\bfx
\end{equation}
where the integration is performed over the first argument of $\mathcal{K}^M$.
Though the slice Fourier transform is defined using the eigenfunction approach, we will demonstrate that it generalises other properties as well.

\subsection{The Mehler construction}
For the functions $\psi_{j,k}$ to be eigenfunctions of the slice Fourier transform, the action of the transform on these functions is to multiply them with proportionality factors which are given by their corresponding eigenvalues $(-i)^{j+k+1}$ (see expression \eqref{formaleigen}). Together with the orthogonality of the Clifford-Hermite functions proved in Theorem \ref{ortho}, this yields the following formal series expansion for the kernel function $\cK^M$:
\begin{align}
\cK^M(\bfx,\bfy)=\sum_{j,k=0}^{+\infty} \frac{\psi_{j,k}(\bfy)(-i)^{j+k+1} \overline{\psi_{j,k}(\bfx)}}{\langle \psi_{j,k},\psi_{j,k} \rangle},\label{KM}
\end{align}
where the denominator accounts for the normalisation of the eigenfunctions. Indeed, after substituting this expression in equation \eqref{general} and changing the order of integration and summation, what appears is the sum of the projections of $f$ onto all $\psi_{j,k}$'s times the slice Fourier transforms of these $\psi_{j,k}$'s. By definition this exactly yields the image of $f$ under the slice Fourier transform.

\subsubsection{Differential and symmetry properties of $\cK^M$}
Before focussing on the explicit calculation of the kernel function, we show that the formal Mehler formula also generalises the differential equations of the classical Fourier transform.

\begin{theorem}
The Mehler formula \eqref{KM} obeys the following system of Clifford-valued partial differential equations:
\begin{align}
\begin{cases}
D^{\bfy}_0 \mathcal{K}^M(\bfx,\bfy) &= -\frac{i}{2c} \mathcal{K}^M(\bfx,\bfy) \bfx \\
i \bfy \mathcal{K}^M(\bfx,\bfy) &=-2c\ [\mathcal{K}^M(\bfx,\bfy) D^{\bfx}_0]
\end{cases}\label{diffeq}
\end{align}
where $[\ .\ D^{\bfx}_0]$ denotes that the differential operator $D_0^{\bfx}$ is acting to the left.
\end{theorem}
\begin{proof}
Both expressions can be proved similarly so we only treat the first relation in full detail. From the definition of the Clifford-Hermite functions $\psi_{j,k}$ and the differential properties of the Clifford-Hermite polynomials $h_{j,k}$, one derives that
\begin{align*}
\psi_{j+1,k}(\bfx) &= \bigg[(\bfx-cD_0^{\bfx})h_{j,k}(\bfx)m_{k}(\bfx)\bigg]\exp\left(-\frac{|\bfx|^2}{4c}\right)\nonumber \\
&=\bfx\psi_{j,k}(\bfx)-c\ C(j,k)\psi_{j-1,k}(\bfx)
\end{align*}
where the real-valued factor $C(j,k)$ is as in Theorem \ref{vglC}.
In Proposition $6.2$ of \cite{Cnudde} the following relation for the Clifford-Hermite functions was obtained:

\begin{align*}
\psi_{j,k}(\bfx) &= \left( \frac{\bfx}{2} - cD_0 \right)\psi_{j-1,k}(\bfx).
\end{align*}
Therefore the action of $D_0^{\bfy}$ on $\mathcal{K}^M$ can be written as
\begin{align*}
&D^{\bfy}_0 \mathcal{K}^M(\bfx,\bfy)\\
&= \sum_{j,k=0}^{\infty} \frac{ \left[D^{\bfy}_0 \psi_{j,k}(\bfy)\right] (-i)^{j+k+1} \overline{\psi_{j,k}(\bfx)}}{\langle \psi_{j,k}, \psi_{j,k} \rangle}\\
&= \sum_{j,k=0}^{\infty} \frac{ \left[\frac{\bfy}{2c} \psi_{j,k}(\bfy) - \frac1c \psi_{j+1,k}(\bfy) \right] (-i)^{j+k+1} \overline{\psi_{j,k}(\bfx)}}{\langle \psi_{j,k}, \psi_{j,k} \rangle}\\
&= \frac1{2c} \sum_{j,k=0}^{\infty} \frac{ \left[\psi_{j+1,k}(\bfy) + c\ C(j,k) \psi_{j-1,k}(\bfy) - 2 \psi_{j+1,k}(\bfy) \right] (-i)^{j+k+1} \overline{\psi_{j,k}(\bfx)}}{\langle \psi_{j,k}, \psi_{j,k} \rangle}\\
&= \frac1{2c} \sum_{j,k=0}^{\infty} \frac{ \left[c\ C(j,k) \psi_{j-1,k}(\bfy) - \psi_{j+1,k}(\bfy) \right] (-i)^{j+k+1} \overline{\psi_{j,k}(\bfx)}}{\langle \psi_{j,k}, \psi_{j,k} \rangle}.
\end{align*}
Changing the summation indices yields
\begin{align*}
D^{\bfy}_0 \mathcal{K}^M(\bfx,\bfy) =& \frac12 \sum_{j,k=0}^{\infty} C(j+1,k) \frac{\psi_{j,k}(\bfy) (-i)^{j+k+2} \overline{\psi_{j+1,k}(\bfx)}}{\langle \psi_{j+1,k}, \psi_{j+1,k} \rangle}\\
&- \frac1{2c} \sum_{j=1,k=0}^{\infty} \frac{\psi_{j,k}(\bfy) (-i)^{j+k} \overline{\psi_{j-1,k}(\bfx)}}{\langle \psi_{j-1,k}, \psi_{j-1,k} \rangle}.
\end{align*}
Based on Theorem \ref{ortho}, one has the following relation for all $j,k \in \mathbb{N}$:
\begin{equation*}
\langle \psi_{j,k}, \psi_{j,k} \rangle = -c\ C(j,k) \langle \psi_{j-1,k}, \psi_{j-1,k} \rangle
\end{equation*}
so $D_0^{\bfy}\mathcal{K}^M(\bfx,\bfy)$ can finally be written as
\begin{align*}
-&\frac{1}{2c} \sum_{j,k=0}^{\infty} \frac{\psi_{j,k}(\bfy) (-i)^{j+k+2} \overline{\psi_{j+1,k}(\bfx)}}{\langle \psi_{j,k}, \psi_{j,k} \rangle} + \frac1{2} \sum_{j,k=0}^{\infty} C(j,k)\frac{\psi_{j,k}(\bfy) (-i)^{j+k} \overline{\psi_{j-1,k}(\bfx)}}{\langle \psi_{j,k}, \psi_{j,k} \rangle}\\
=&\frac{i}{2c} \sum_{j,k=0}^{\infty} \frac{\psi_{j,k}(\bfy) (-i)^{j+k+1} \overline{[\psi_{j+1,k}(\bfx)+c\ C(j,k)\psi_{j-1,k}(\bfx)]}}{\langle \psi_{j,k}, \psi_{j,k} \rangle}\\
=&\frac{i}{2c} \sum_{j,k=0}^{\infty} \frac{\psi_{j,k}(\bfy) (-i)^{j+k+1} \overline{\bfx\psi_{j,k}(\bfx)}}{\langle \psi_{j,k}, \psi_{j,k} \rangle}\\
=&-\frac{i}{2c} \mathcal{K}^M(\bfx,\bfy) \bfx,
\end{align*}
which proves the first relation. The proof of the second relation is completely analogous when starting from the right-hand side.
\end{proof}
Furthermore, equation \eqref{KM} exhibits an interesting symmetry property. By the non-commutativity of $\bfx$ and $\bfy$, however, it will only show up with respect to an anti-involution. Indeed, taking the Clifford conjugation of \eqref{KM} yields
\begin{align*}
\overline{\mathcal{K}^M(\bfx,\bfy)}&= \overline{ \sum_{j,k=0}^{\infty} \frac{\psi_{j,k}(\bfy)(-i)^j \overline{\psi_{j,k}(\bfx)}}{\langle \psi_{j,k}, \psi_{j,k} \rangle}}\\
&=\sum_{j,k=0}^{\infty} \frac{\psi_{j,k}(\bfx)\overline{(-i)^j} \overline{\psi_{j,k}(\bfy)}}{\langle \psi_{j,k}, \psi_{j,k} \rangle}=\left(\mathcal{K}^M(\bfy,\bfx)\right)^*.
\end{align*}

\begin{remark} The symmetry of $\cK^M$ allows to demonstrate the equivalence of the two partial differential equations. Indeed, from $\overline{\mathcal{K}^M(\bfx,\bfy)}=\left(\mathcal{K}^M(\bfy,\bfx)\right)^*$ it follows that 
\begin{align*}
\overline{D^{\bfy}_0 \mathcal{K}^M(\bfx,\bfy)} &= \overline{\frac{-i}{2c} \mathcal{K}^M(\bfx,\bfy) \bfx}\\
\Leftrightarrow \left[ \overline{\mathcal{K}^M(\bfx,\bfy)}^{\phantom{\mC}}\ \overline{D^{\bfy}_0} \right] &= \frac{i}{2c} \overline{\bfx} \overline{\mathcal{K}^M(\bfx,\bfy)}\\
\Leftrightarrow \left[ \left(\mathcal{K}^M(\bfy,\bfx)\right)^*\ D^{\bfy}_0\right]  &= \frac{i}{2c} \bfx \left(\mathcal{K}^M(\bfy,\bfx)\right)^*\\
\Leftrightarrow
\left[ \mathcal{K}^M(\bfx,\bfy)^{\phantom{\mC}}\ D^{\bfx}_0 \right] &= \frac{-i}{2c} \bfy \mathcal{K}^M(\bfx,\bfy),
\end{align*}
which equals the second equation. In the last step the complex conjugate is taken and the variables $\bfx$ and $\bfy$, regarded as dummy variables, are interchanged.
\end{remark}

\subsubsection{Closed form of the kernel function $\cK^M$}
Because all entities in \eqref{KM} are known, their explicit expressions can be substituted and a closed expression for $\cK^M$ can be obtained. The explicit form of the Clifford-Hermite functions $\psi_{j,k}$ depends on the parity of their first index $j$. Using the Laguerre form (see Theorem \ref{Laguerre}) for the Clifford-Hermite polynomials, one has
\begin{align*}
\psi_{2t,k}(\bfx)&=(2c)^tt!\ L_t^k\left(\frac{|\bfx|^2}{2c}\right)(e_0-1)(x_0+\ux)^k \exp\left(-\frac{|\bfx|^2}{4c}\right)\\
\psi_{2t+1,k}(\bfx)&=(2c)^tt!\ \bfx\ L_t^{k+1}\left(\frac{|\bfx|^2}{2c}\right)(e_0-1)(x_0+\ux)^k \exp\left(-\frac{|\bfx|^2}{4c}\right)
\end{align*}
and their norms, who show up in the denominator of \eqref{KM}, are given in Theorem \ref{ortho}. After direct substitution of these expressions, one obtains
\begin{align*}
\mathcal{K}^M(\bfx,\bfy) =&\frac{-i\Gamma\left(\frac{m}{2}\right)}{2\pi^{m/2+1}}\ \exp\left(-\frac{|\bfx|^2 + |\bfy|^2}{4c} \right) \\
&\times \left[ \phantom{+i\ }\sum_{k=0}^{+\infty} \frac{\phantom{\bfy}(1-e_0)(y_0+\uy)^k (-i)^k(x_0-\ux)^k(e_0+1)\phantom{\bfx}}{(2c)^{k+1}} H_t^k(\bfx,\bfy) \right.\\
&\left. \phantom{+i}+\ i \sum_{k=0}^{+\infty} \frac{\bfy(1-e_0)(y_0+\uy)^k (-i)^k(x_0-\ux)^k(e_0+1)\bfx}{(2c)^{k+2}} H_t^{k+1}(\bfx,\bfy)\right].
\end{align*}
In each summation, all $t$-dependencies are real-valued and could therefore be grouped in a separate factor $H_t^k$ which reads
\begin{align*}
H_t^k(\bfx,\bfy)= \sum_{t=0}^{+\infty}\frac{(-1)^t t!\ L_t^k\left(\frac{|\bfy|^2}{2c}\right)L_t^k\left(\frac{|\bfx|^2}{2c}\right)}{(k+t)!}
\end{align*}
and is equal for both summations, apart from interchanging $k\leftrightarrow k+1$. The latter expression constitutes a special case of the so-called Hille-Hardy formula \cite{szego}, which reads
\begin{align*}
\sum_{t=0}^{+\infty} \frac{t!}{\Gamma(k+t+1)} L_t^{k}(x) L_t^{k}(y) z^t = \frac{(xyz)^{-\frac{k}{2}}}{(1-z)} \exp\left(-\frac{(x+y)z}{1-z}\right) I_{k}\left(\frac{2\sqrt{xyz}}{1-z}\right)
\end{align*}
for $|z| < 1$ and where the modified Bessel function $I_k$ obeys $I_k(x)=i^{-k}J_k(ix)$. The fact that $H_t^k$ is just outside the general domain of convergence is compensated for by the heuristic character of this reasoning. The validity of the final expression for $\cK^M$ will be verified explicitly afterwards in Theorem \ref{eig}.\\
Moreover, recalling that $\bfx=x_0e_0+\ux$ and $\bfy=y_0e_0+\uy$ these variables obey the identities $ (e_0+1)\bfx=(x_0+\ux)(e_0-1)$ and $\bfy (1-e_0)=(1+e_0)(y_0+\uy)$. Thereby $\mathcal{K}^M$ is proportional to
\begin{align*}
& \phantom{-\ }(1-e_0) \sum_{k=0}^{+\infty} \frac{(y_0+\uy)^{k} (-i)^{k}(x_0-\ux)^{k}}{(|\bfx||\bfy|)^k}J_{k}\left(\frac{|\bfx||\bfy|}{2c}\right)(e_0+1) \\
& - (1+e_0)\sum_{k=0}^{+\infty} \frac{(y_0+\uy)^{k+1} (-i)^{k+1} (x_0-\ux)^{k+1}}{(|\bfx||\bfy|)^{k+1}} J_{k+1}\left(\frac{|\bfx||\bfy|}{2c}\right)(e_0-1) .
\end{align*}
Terms corresponding to Bessel functions of the same order can now be gathered by changing the summation index in the second sum, yielding
\begin{align}
&2 J_0\left(\frac{|\bfx||\bfy|}{2c}\right) \nonumber \\
&+ 2 \sum_{k=1}^{+\infty} (-i)^k \frac{(y_0+\uy)^k (x_0-\ux)^k + (y_0-\uy)^k (x_0+\ux)^k}{(|\bfx||\bfy|)^{k}}  J_k\left(\frac{|\bfx||\bfy|}{2c}\right).\label{besselsum}
\end{align}
At this point, we wish to use the following property of an infinite sum of products of Bessel and cosine functions \cite{watson}:
\begin{align}
e^{-iz \cos(\phi)}&=J_0(z)+2\sum_{n=1}^{+\infty}(-i)^n J_n(z) \cos(n\phi).\label{bessel}
\end{align}
In order to retrieve this structure in our series expression, we introduce polar coordinates for $\bfx$ and $\bfy$, respectively in the $(x_0,r)$- and the $(y_0,g)$-plane, using the respective angles $\chi$ and $\phi$:
\begin{align*}
\frac{x_0+\ux}{\sqrt{x_0^2+r^2}} &= \cos(\chi) + \uo \sin(\chi)\\
\frac{y_0+\uy}{\sqrt{y_0^2+g^2}} &= \cos(\phi) + \ueta \sin(\phi).
\end{align*}
Regarding $\uo$ and $\ueta$ as counterparts of the classical complex unit, these definitions can be split in their real and imaginary parts. Substituting the latter in \eqref{besselsum}, $\cK^M$ turns out to be proportional to
\begin{align*}
2 J_0\left(\frac{|\bfx||\bfy|}{2c}\right)
&+ 2 \sum_{k=1}^{+\infty} (-i)^k \left(\cos\left(k(\phi+\chi)\right) +\cos\left(k(\phi-\chi)\right) \right) J_k\left(\frac{|\bfx||\bfy|}{2c}\right)\\
& - 2 \sum_{k=1}^{+\infty} (-i)^k \left( \cos\left(k(\phi-\chi)\right) - \cos\left(k(\phi+\chi)\right) \right) J_k\left(\frac{|\bfx||\bfy|}{2c}\right) \ueta \uo
\end{align*}
where the appropriate trigonometric formulas have been used. Recalling equation \eqref{bessel}, the previous expression can still be rewritten as
\begin{align*}
\mathcal{K}^M(\bfx,\bfy) =\frac{-i\Gamma\left(\frac{m}{2}\right)}{8c\pi^{m/2+1}} &\left[ \phantom{\ }\left(e^{-\frac{i}{2c}|\bfx||\bfy|\cos(\phi+\chi)} + e^{-\frac{i}{2c}|\bfx||\bfy|\cos(\phi-\chi)}\right) \right.\\
&\left. - \left( e^{-\frac{i}{2c}|\bfx||\bfy|\cos(\phi-\chi)} - e^{-\frac{i}{2c}|\bfx||\bfy|\cos(\phi+\chi)} \right) \ueta\uo \right].
\end{align*}
Finally, eliminating the angular variables and grouping corresponding terms, we obtain
\begin{align}
\mathcal{K}^M(\bfx,\bfy) =\frac{-i\Gamma\left(\frac{m}{2}\right)}{8c\pi^{m/2+1}} \left[ (1+\ueta\uo) e^{-\frac{i}{2c}(x_0y_0-rg)} + (1-\ueta\uo) e^{-\frac{i}{2c}(x_0y_0+rg)} \right].\label{kernel}
\end{align}
A straightforward verification shows that \eqref{kernel} is indeed a solution of the Clifford-valued partial differential equations
\begin{align*}
\begin{cases}
D^{\bfy}_0 \mathcal{K}^M(\bfx,\bfy) &= -\frac{i}{2c} \mathcal{K}^M(\bfx,\bfy) \bfx\\
i \bfy \mathcal{K}^M(\bfx,\bfy) &=-2c\ [\mathcal{K}^M(\bfx,\bfy) D^{\bfx}_0]
\end{cases}
\end{align*}
and also obeys $\overline{\cK^M(\bfx,\bfy)}=\left(\cK^M(\bfy,\bfx)\right)^*$. Having obtained this closed expression for the kernel function, we can conclude this subsection with the explicit definition of the slice Fourier transform.
\begin{definition}[Slice Fourier transform]\label{sFT} The slice Fourier transform of a function $f\in \mathcal{V}$ is given by
\begin{align*}
\mathcal{F}_S(f)(\bfy) =&\frac{-i\Gamma\left(\frac{m}{2}\right)}{8c\pi^{m/2+1}}\times \\ \displaystyle\int\limits_{\mR^{m+1}}& \left[ (1+\ueta\uo) e^{-\frac{i}{2c}(x_0y_0-rg)} + (1-\ueta\uo) e^{-\frac{i}{2c}(x_0y_0+rg)} \right] f(\bfx)\ \mathrm{d}x_0\ \mathrm{d}r\ \mathrm{d}\sigma_{\ux}
\end{align*}
with $\bfx=x_0e_0+r\uo$ and $\bfy=y_0e_0+g\ueta$.
\end{definition}

\begin{remark}
Using Eulers formula, the following equivalent expression for the slice Fourier transform is obtained
\begin{align*}
\mathcal{F}_S(f)(\bfy) =&\frac{-i\Gamma\left(\frac{m}{2}\right)}{4c\pi^{m/2+1}}\displaystyle\int\limits_{\mR^{m+1}} e^{-\frac{i}{2c}(x_0y_0)} \left[ \cos\left(\frac{rg}{2c}\right) + i\ueta\uo \sin\left(\frac{rg}{2c}\right) \right] f(\bfx) \mathrm{d}x_0 \mathrm{d}r \mathrm{d}\sigma_{\ux}.
\end{align*}
\end{remark}

\subsection{Verification of the eigenfunctions}
Now the integral expression for the slice Fourier transform has been obtained explicitly, we can further investigate its properties. In this first subsection we demonstrate that definition \ref{sFT} is indeed an integral transform whose eigenfunctions are the Clifford-Hermite functions.
\begin{theorem}\label{eig}The functions $\psi_{j,k}$ are eigenfunctions of the integral transform $\mathcal{F}_S$ with eigenvalues $(-i)^{j+k+1}$.
\end{theorem}
\begin{proof}
We already proved that the kernel function $\mathcal{K}^M$ obeys the system of Clifford-valued partial differential equations
\begin{empheq}[left={\empheqlbrace}]{align}
D^{\bfy}_0 \mathcal{K}^M(\bfx,\bfy) &= -\frac{i}{2c} \mathcal{K}^M(\bfx,\bfy) \bfx \label{D1}\\
i \bfy \mathcal{K}^M(\bfx,\bfy) &=-[2c\ \mathcal{K}^M(\bfx,\bfy) D^{\bfx}_0].\label{D2}
\end{empheq}
Using these equations and the fact that $\psi_{j,k}(\bfx)=\left(\frac{\bfx}{2}-cD_0^{\bfx}\right)^j\psi_{0,k}(\bfx)$, one immediately gets
\begin{align*}
\mathcal{F}_S(\psi_{j,k})(\bfy)&=\int_{\mR^{m+1}} \mathcal{K}^M (\bfx,\bfy) \left(\frac{\bfx}{2}-cD_0^{\bfx}\right)^j\psi_{0,k}(\bfx)\ \mathrm{d}x_0\ \mathrm{d}r\ \mathrm{d}\sigma_{\ux}\\
&=(-i)^j\left(\frac{\bfy}{2}-cD_0^{\bfy}\right)^j \int_{\mR^{m+1}} \mathcal{K}^M (\bfx,\bfy) \psi_{0,k}(\bfx)\ \mathrm{d}x_0\ \mathrm{d}r\ \mathrm{d}\sigma_{\ux},
\end{align*}
so we only have to prove that the functions $\psi_{0,k}$ are eigenfunctions of $\mathcal{F}_S$ with corresponding eigenvalues $(-i)^{k+1}$. Using the equality
\begin{equation*}(x_0+\ux)=\left(\frac{1-e_0}{\sqrt{2}}\right) \bfx \left(\frac{1-e_0}{\sqrt{2}}\right),\end{equation*} the expression for the Hermite functions $\psi_{0,k}$ can be rewritten as
\begin{align*}
\psi_{0,k}(\bfx)
&=(e_0-1)\left[\left(\frac{1-e_0}{\sqrt{2}}\right) \bfx \left(\frac{1-e_0}{\sqrt{2}}\right)\right]^k \exp\left(-\frac{|\bfx|^2}{4c}\right).
\end{align*}
Because $\mathcal{K}^Me_0=e_0\mathcal{K}^M$ and using \eqref{D1}, the slice Fourier transform of $\psi_{0,k}$ is given by
\begin{align*}
& \int\displaylimits_{\mR^{m+1}} \mathcal{K}^M (\bfx,\bfy) (e_0-1)\left[\left(\frac{1-e_0}{\sqrt{2}}\right) \bfx \left(\frac{1-e_0}{\sqrt{2}}\right)\right]^k \exp\left(-\frac{|\bfx|^2}{4c}\right) \ \mathrm{d}x_0\ \mathrm{d}r\ \mathrm{d}\sigma_{\ux}\\
&=(e_0-1)(2ci)^k\left[\left(\frac{1-e_0}{\sqrt{2}}\right) D_0^{\bfy} \left(\frac{1-e_0}{\sqrt{2}}\right)\right]^k\\
&\hspace{32ex}\times  \int\displaylimits_{\mR^{m+1}} \mathcal{K}^M (\bfx,\bfy) \exp\left(-\frac{|\bfx|^2}{4c}\right)\ \mathrm{d}x_0\ \mathrm{d}r\ \mathrm{d}\sigma_{\ux}.
\end{align*}
Using the fact that $\int_{\mathbb{S}^m}\uo\ \mathrm{d}\sigma_{\ux}=0$ together with the identities
\begin{align*}
\int\displaylimits_{-\infty}^{+\infty} e^{\pm i\frac{x_0y_0}{2c}}e^{-\frac{x_0^2}{4c}} \mathrm{d}x_0 = 2\sqrt{\pi c}\ e^{-\frac{y_0^2}{4c}}\text{\ \ and \ } \int\displaylimits_{0}^{+\infty} \cos\left( \frac{rg}{2c} \right) e^{-\frac{r^2}{4c}} \mathrm{d}r = \sqrt{\pi c}\ e^{-\frac{g^2}{4c}},
\end{align*}
the integral reduces to 
\begin{align*}
\frac{2}{4\pi c}\int\displaylimits_{-\infty}^{+\infty}\int\displaylimits_{0}^{+\infty} e^{-\frac{i}{2c}x_0y_0}\cos\left(\frac{rg}{2}\right) \exp\left(-\frac{|\bfx|^2}{4c}\right)\ \mathrm{d}x_0\mathrm{d}r = \exp\left(-\frac{|\bfy|^2}{4c}\right)
\end{align*}
and we finally obtain
\begin{align*}
\mathcal{F}_S(\psi_{0,k})(\bfy) &=(2ci)^k(e_0-1)\left[\left(\frac{1-e_0}{\sqrt{2}}\right)D_0^{\bfy} \left(\frac{1-e_0}{\sqrt{2}}\right)\right]^k \exp\left(-\frac{|\bfy|^2}{4c}\right)\\
&=-i(2ci)^k(e_0-1)\left(\partial_{y_0}+\ueta\partial_g\right)^k \exp\left(-\frac{|\bfy|^2}{4c}\right)\\
&=-i\frac{(2ci)^k}{(-2c)^k}(e_0-1)(y_0+\uy)^k \exp\left(-\frac{|\bfy|^2}{4c}\right)\\
&=(-i)^{k+1}\psi_{0,k}(\bfy).
\end{align*}
As was noticed above, the order $j$ yields another factor $(-i)^j$ so we have proven that the functions $\psi_{j,k}$ are eigenfunctions of $\mathcal{F}_S$ with eigenvalues $(-i)^{j+k+1}$.
\end{proof}

\subsection{Verification of the action of the inverse}
In the derivation of the integral kernel, we required the eigenvalues of the functions $\psi_{j,k}$ to be $(-i)^{j+k+1}$. Now the action of the inverse transform is to undo the action of the slice Fourier Transform, so to transform $(-i)^{j+k+1}\psi_{j,k}$ into the original Clifford-Hermite function $\psi_{j,k}$. Therefore we require the eigenvalues of the inverse transform with respect to $\psi_{j,k}$ to be $i^{j+k+1}$. Given the formal series expression \eqref{KM} of $\cK^M$, one quickly observes that by this reasoning the kernel of the inverse transform will be the complex conjugate of the kernel of the slice Fourier Transform. We thus get the following theorem.
\begin{theorem}\label{inverstrans}The functions $\psi_{j,k}$ are eigenfunctions of the integral transform
\begin{align*}
\mathcal{F}_S^{-1}(F)(\bfy) =&\frac{i\Gamma\left(\frac{m}{2}\right)}{8c\pi^{m/2+1}} \\
\int\displaylimits_{\mR^{m+1}} &\left[ (1+\ueta\uo) e^{\frac{i}{2c}(x_0y_0-rg)} + (1-\ueta\uo) e^{\frac{i}{2c}(x_0y_0+rg)} \right] F(\bfx) \mathrm{d}x_0 \mathrm{d}r \mathrm{d}\sigma_{\ux}
\end{align*}
with eigenvalues $i^{j+k+1}$, where $\bfx=x_0e_0+r\uo$ and $\bfy=y_0e_0+g\ueta$.
\end{theorem}
\begin{proof}
The kernel function of the inverse slice Fourier transform, denoted as $\cK^{M,-1}$, is the complex conjugate of $\cK^M$. Therefore we obtain from equation \eqref{diffeq} that it obeys the following system of Clifford-valued partial differential equations:
\begin{align*}
\begin{cases}
D^{\bfy}_0 \mathcal{K}^{M,-1}(\bfx,\bfy) &= \frac{i}{2c} \mathcal{K}^{M,-1}(\bfx,\bfy) \bfx\\
i \bfy \mathcal{K}^{M,-1}(\bfx,\bfy) &=[2c\ \mathcal{K}^{M,-1}(\bfx,\bfy) D^{\bfx}_0].
\end{cases}
\end{align*}
The rest of the proof is completely analogous to the proof of Theorem \ref{eig}.
\end{proof}

\begin{theorem}
The integral transform $\mathcal{F}_S^{-1}$ in the above Theorem \ref{inverstrans} is the inverse slice Fourier transform of a function $F\in\mathcal{V}$.
\end{theorem}
\begin{proof}
Using Theorem \ref{eig} and Theorem \ref{inverstrans} we get for all Clifford-Hermite functions $\psi_{j,k}$ that
\begin{align*}
(\cF_S^{-1}\circ \cF_S)(\psi_{j,k}) = (\cF_S^{-1})((-i)^{j+k+1}\psi_{j,k}) = (i)^{j+k+1}(-i)^{j+k+1}\psi_{j,k}=\psi_{j,k}
\end{align*}
and analogously for $(\cF_S\circ \cF_S^{-1})(\psi_{j,k})$. Because $F\in\text{span}\{\psi_{j,k}\}$ this proves the theorem.
\end{proof}

\subsection{Basic properties of the slice Fourier transform}
Within this subsection the constant prefactor of the trigonometric expression for $\cK^M$ will be denoted as
\begin{equation*}
C_m=\frac{-i\Gamma\left(\frac m2 \right)}{4c\pi^{m/2+1}}.
\end{equation*} 
As announced, the explicit expression for the integral kernel allows for the study of various properties of the slice Fourier transform by direct calculation. Here some of these basic properties are summarised.
\subsubsection{Translation property}
Denoting a translation in the $e_0$-direction as $t_af(x_0,r,\uo) = f(x_0-a,r,\uo)$, one has
\begin{align*}
\mathcal{F}_S(t_af)(\bfy) &= \int\displaylimits_{\mR^{m+1}}\mathcal{K}(\bfx,\bfy)f(x_0-a,r,\uo) \mathrm{d}x_0\mathrm{d}r\mathrm{d}\sigma_{\ux} \\
&= e^{-\frac{iay_0}{2c}} \mathcal{F}_S(f)(\bfy).
\end{align*}
There is no analogous property for translations in the $\uo$-direction because of its spherical nature.
\subsubsection{Reflection property}
Denoting a reflection with respect to the origin as $sf(x_0,r,\uo) = f(-x_0,r,-\uo)$, one has
\begin{align*}
\mathcal{F}_S(sf)(\bfy) &= C_m\ \int\displaylimits_{\mR^{m+1}} e^{\frac{ix_0y_0}{2c}} \left[ \cos\left( \frac{rg}{2c}\right) -\ueta\uo i \sin\left(\frac{rg}{2c}\right) \right] f(x_0,r,\uo) \mathrm{d}x_0\mathrm{d}r\mathrm{d}\sigma_{\ux} \\
&= \mathcal{F}_S(f)(-y_0,g,-\ueta) \\
&= s \mathcal{F}_S(f)(\bfy).
\end{align*}

\subsubsection{Complex conjugation property}
Denoting the complex conjugate of $f$ as before as $f^*$, one has
\begin{align*}
\mathcal{F}_S\left(f^*\right)(\bfy) &= C_m \int\displaylimits_{\mR^{m+1}} e^{-\frac{ix_0y_0}{2c}} \left[ \cos\left( \frac{rg}{2c}\right) +\ueta\uo i \sin\left(\frac{rg}{2c}\right) \right]f^* \mathrm{d}x_0\mathrm{d}r\mathrm{d}\sigma_{\ux}\\
&= -C_m^* \int\displaylimits_{\mR^{m+1}} \left(e^{-\frac{ix_0(-y_0)}{2c}} \left[ \cos\left( \frac{rg}{2c}\right) +(-\ueta)\uo i \sin\left(\frac{rg}{2c}\right) \right]f(\bfx)\right)^* \mathrm{d}x_0\mathrm{d}r\mathrm{d}\sigma_{\ux}\\
&= -\mathcal{F}^*_S(f)(-y_0,g,-\ueta).
\end{align*}

\subsubsection{Commutation with $e_0$}
Given that both $\uo$ and $\ueta$ anticommute with $e_0$, one has
\begin{align*}
\mathcal{F}_S(e_0 f)(\bfy)&= e_0\mathcal{F}_S(f)(\bfy).
\end{align*}

\subsubsection{Twofold transform}
Applying two consecutive slice Fourier transforms to a function $f\in\text{span}\{\psi_{j,k}\}$ yields
\begin{align*}
\mathcal{F}_S\left( \mathcal{F}_S(f) \right)(\bfx) = -f(-\bfx)
\end{align*}
because
$\mathcal{F}_S\left( \mathcal{F}_S(\psi_{j,k}) \right)(\bfx) = (-i)^{2+2j+2k} \psi_{j,k}(\bfx) = (-1)^{j+k+1} \psi_{j,k}(\bfx)$, which equals $-\psi_{j,k}(-\bfx)$ because $\psi_{j,k}(-\bfx)=(-1)^{j+k}\psi_{j,k}(\bfx)$.

\subsection{Explicit calculation}\label{explicit}
In this last subsection, we take a closer look at the computational load of the slice Fourier transform. To this end we consider the computation of the slice Fourier transform of a function $f\in\mathcal{V}$ explicitly. Using Remark \ref{splits}
 and performing the spherical integration in the definition of the transform, we obtain
\begin{align*}
\mathcal{F}_{S}(f)(\bfy) =\frac{-i}{4\pi c} \int\displaylimits_{-\infty}^{+\infty} \int\displaylimits_{0}^{+\infty} & \left[\phantom{-}\left( e^{-\frac{i}{2c}(x_0y_0-rg)} + e^{-\frac{i}{2c}(x_0y_0+rg)} \right) f_1(x_0,r)\right.\\
&\phantom{\big(}-\left. \left( e^{-\frac{i}{2c}(x_0y_0-rg)} - e^{-\frac{i}{2c}(x_0y_0+rg)} \right) \ueta f_2(x_0,r)\right] \mathrm{d}x_0\ \mathrm{d}r,
\end{align*}
which equals
\begin{align*}
\mathcal{F}_{S}(f)(\bfy) =&\frac{-i}{2\pi c} \int\displaylimits_{-\infty}^{+\infty} \int\displaylimits_{0}^{+\infty} e^{-\frac{i}{2c}x_0y_0} \cos\left(\frac{rg}{2c}\right) f_1(x_0,r)\ \mathrm{d}x_0\ \mathrm{d}r\\
&\hspace{17ex}-\ueta \frac{1}{2\pi c} \int\displaylimits_{-\infty}^{+\infty} \int\displaylimits_{0}^{+\infty} e^{-\frac{i}{2c}x_0y_0} \sin\left(\frac{rg}{2c}\right) f_2(x_0,r)\ \mathrm{d}x_0\ \mathrm{d}r.
\end{align*}
Extending $f_1$ to the function $f_1^+:\mR^2\rightarrow Cl^{m+1}: (x_0,r)\mapsto f_1(x_0,|r|)$ which is even in its second argument and analogously extending $f_2$ to the function
\begin{align*}
f_2^-:\mR^2\rightarrow Cl^{m+1}:(x_0,r)\mapsto \begin{cases} \phantom{-}f_2(x_0,\phantom{-}r) \qquad r>0 \\ -f_2(x_0,-r) \qquad r<0 \end{cases}
\end{align*}
which is odd in its second argument, the integrals can be rewritten as
\begin{align*}
\mathcal{F}_{S}(f)(\bfy) =&\frac{-i}{4\pi c} \int\displaylimits_{-\infty}^{+\infty} \int\displaylimits_{-\infty}^{+\infty} e^{-\frac{i}{2c}(x_0y_0+rg)} \left[f_1^+(x_0,r) -i \ueta f_2^-(x_0,r)\right] \mathrm{d}x_0\ \mathrm{d}r.
\end{align*}
Therefore we can conclude that the slice Fourier transform $\mathcal{F}_S(f)$ of a (complex) Clifford-valued function $f$ can be calculated solely using the classical two-dimensional Fourier transform of a single two-dimensional function $F(x_0,r)=f_1^+(x_0,r) -i f_2^-(x_0,r)$:
\begin{align*}
\mathcal{F}_{S}(f)(\bfy) =&\frac{-i}{2 c} \left( FT(F)^+(y_0,g) + \ueta\ FT(F)^-(y_0,g)  \right) 
\end{align*}
Again the superscripts $+$ and $-$ denote, respectively, the even and odd parts of the functions in their second argument and $FT$ is the classical two-dimensional Fourier transform of a complex-valued function $h$ defined as
\begin{align*}
FT(h)(y_0,g)= \frac{1}{2\pi} \int\displaylimits_{-\infty}^{+\infty} \int\displaylimits_{-\infty}^{+\infty} e^{-\frac{i}{2c}(x_0y_0+rg)} h(x_0,r)\ \mathrm{d}x_0\ \mathrm{d}r.
\end{align*}

\section{Convolutions}
In this section two different approaches to the convolution operator associated with the slice Fourier transform are treated. Both of them satisfy the demand that the slice Fourier transform of them equals the product of the separate Fourier transforms. In a first approach the classical definition of the convolution is generalised using the Mustard convolution (see \cite{Mustard}). In a second approach the translation operator in the integrand of the convolution is generalised, based on \cite{Convolutions}. Finally, given that both approaches give rise to the same behaviour in the Fourier domain, we pinpoint the connection between both in the last subsection.

A major drawback in constructing convolutions in a Clifford setting, however, is the lack of commutativity. Any Clifford counterpart to the classical convolution will therefore automatically fail to generalise one of its primary properties, namely $f\star g=g\star f$.

\begin{remark}
In this section all generalisations are constructed such that in the Fourier domain the function $f$ is right multiplied with $g$. Analogous constructions can be done for left multiplication.
\end{remark}

\subsection{Technical prerequisites}
\noindent Let us first recall the following lemma which computes the integral of a squared component of a unit vector over the unit sphere.
\begin{lemma}\label{omegai}
With $\uo = \omega_1 e_1 + \omega_2 e_2 + \ldots + \omega_{m} e_{m}$ a unit vector in $Cl_{m}$, one has 
\begin{align*}
\int\displaylimits_{\mathbb{S}^{m-1}} \omega_i^2\mathrm{d}\sigma_{\uo}&= \frac{2\pi^{m/2}}{m\Gamma\left(\frac{m}{2}\right)}.
\end{align*}
\end{lemma}

\begin{proof}
Given that 
\begin{align*}
\int\displaylimits_{\mathbb{S}^{m-1}} \uo^2\mathrm{d}\sigma_{\uo}=
-\int\displaylimits_{\mathbb{S}^{m-1}} 1\ \mathrm{d}\sigma_{\uo}=
-\sum_{i=1}^m\int\displaylimits_{\mathbb{S}^{m-1}} \omega_i^2\mathrm{d}\sigma_{\uo},
\end{align*}
the above expression follows from the symmetry of the problem and the expression for the area of an $m$-dimensional sphere, $\frac{2\pi^{m/2}}{\Gamma\left(\frac{m}{2}\right)}$.
\end{proof}

\noindent Another useful lemma addresses the sum of a $k$-vector squeezed between all basis elements of $Cl_m$.
\begin{lemma}\label{eae}
With $\underline{a}^{(k)}$ a $k$-vector in $Cl_m$ (so $k\leq m$), one has 
\begin{align*}
\sum_{i=1}^m e_i\underline{a}^{(k)}e_i = (-1)^k(2k-m) \underline{a}^{(k)}.
\end{align*}
\end{lemma}

\begin{proof}
Given that $a^{(k)}$ is a $k$-vector, possible values for $k$ range from $0$ to $m$. We prove this lemma using induction.
\begin{itemize}
\item $\mathbf{k=0}$ When $\underline{a}^{(k)}$ is a scalar the expression reads $$\sum_{i=1}^m e_i\underline{a}^{(0)}e_i =-m\underline{a}^{(0)},$$ which is identically true.
\item $\mathbf{k=k'+1\ (\leq m)}$ Assuming the validity for $k=k'\ (k'<m)$, one has $\sum_{i=1}^m e_i\underline{a}^{(k')}e_i = (-1)^{k'}(2k'-m) \underline{a}^{(k')}$. Without loss of generality we assume $\underline{a}^{(k)}$ to consist of one single term. Now a $(k'+1)$-vector can be constructed as $\underline{a}^{(k'+1)}=\underline{a}^{(k')}e_{\ell}$ where the basis vector $e_{\ell}$ is not contained in $\underline{a}^{(k')}$. We get
\begin{align*}
\sum_{i=1}^m e_i\underline{a}^{(k'+1)}e_i &= \sum_{i=1}^m e_i\underline{a}^{(k')}e_{\ell}e_i\\
&=\sum_{i=1}^m e_i\underline{a}^{(k')}(-2\delta_{i\ell}-e_ie_{\ell})\\
&=-2e_{\ell}\underline{a}^{(k')} - \left(\sum_{i=1}^m e_i\underline{a}^{(k')}e_i\right) e_{\ell}\\
&=-2(-1)^{k'}\underline{a}^{(k')}e_{\ell} - (-1)^{k'}(2k'-m) \underline{a}^{(k')}e_{\ell}\\
&=(-1)^{k'+1}[2 + (2k'-m) ] \underline{a}^{(k')}e_{\ell}\\
&=(-1)^{k'+1}(2(k'+1)-m) \underline{a}^{(k'+1)},
\end{align*}
which proves the expression for $k=k'+1$.
\end{itemize}
By induction, this proves the lemma.
\end{proof}

\noindent Now we can state the main technical theorem. The integral treated by this theorem will be met several times in the following subsections.
\begin{theorem}\label{ak}
With $\underline{a}\in Cl_{m}$, one has 
\begin{align*}\int\displaylimits_{\mathbb{S}^{m-1}} \uo\ \underline{a}\ \uo\ \mathrm{d}\uo = \frac{2\pi^{m/2}}{m\Gamma\left(\frac{m}{2}\right)}\ \sum_{k=0}^m (-1)^k(2k-m) \underline{a}^{(k)}.
\end{align*}
\end{theorem}

\begin{proof}
Writing
\begin{align*}\uo =\sum_{i=1}^m \omega_i e_i,\end{align*}
the integral can be rewritten as
\begin{align*}\int\displaylimits_{\mathbb{S}^{m-1}} \uo\ \underline{a}\ \uo\ \mathrm{d}\uo = \sum_{i,j=1}^m \left( \int\displaylimits_{\ \mathbb{S}^{m-1}} \omega_i\omega_j\mathrm{d}\sigma_{\uo} \right) e_i\underline{a}e_j 
= \sum_{i=1}^m \left( \int\displaylimits_{\ \mathbb{S}^{m-1}} \omega_i^2\mathrm{d}\sigma_{\uo} \right) e_i\underline{a}e_i 
\end{align*}
where the last equality holds because of symmetry reasons. Using lemmas \ref{omegai} and \ref{eae}, the theorem follows.
\end{proof}

\subsection{Mustard convolution}\label{sectie2}
As pointed out in the introduction, a Mustard convolution is defined such that its Fourier transform equals the product of the Fourier transforms of both functions separately. The Mustard convolution corresponding to the slice Fourier transform is therefore defined as
\begin{align*}
f \star_S g = \mathcal{F}_S^{-1}\left(\mathcal{F}_S(f) \mathcal{F}_S(g)\right),
\end{align*}
where $f$ and $g$ are $Cl_{m+1}$-valued functions belonging to $\mathcal{V}$. Using $\bfy=y_0e_0+g\ueta$, $\bfz=z_0e_0+n\uzeta$ and $\bfu=u_0e_0+h\unu$ as integration variables, the right-hand side reads in full
\begin{align*}
\mathcal{F}_S^{-1}\left(\mathcal{F}_S(f) \mathcal{F}_S(g)\right)&(\bfx)=-i\left(\frac{\Gamma\left(\frac{m}{2}\right)}{8c\pi^{m/2+1}}\right)^3 \int\displaylimits_{\mR^{m+1}} \int\displaylimits_{\mR^{m+1}} \int\displaylimits_{\mR^{m+1}} \\
 &\left[ (1+\uo\ueta) e^{+\frac{i}{2c}(x_0y_0-rg)} + (1-\uo\ueta) e^{+\frac{i}{2c}(x_0y_0+rg)} \right] \times \\
& \left[ (1+\ueta\uzeta)\hspace{0.3ex} e^{-\frac{i}{2c}(y_0z_0-gn)} + (1-\ueta\uzeta)\hspace{0.3ex} e^{-\frac{i}{2c}(y_0z_0+gn)} \right] f(\bfz)\ \times \\
& \left[ (1+\ueta\unu) e^{-\frac{i}{2c}(y_0u_0-gh)} + (1-\ueta\unu) e^{-\frac{i}{2c}(y_0u_0+gh)} \right] g(\bfu)\ \times \\
& \mathrm{d}y_0\ \mathrm{d}g\ \mathrm{d}\sigma_{\uy}\ \mathrm{d}z_0\ \mathrm{d}n\ \mathrm{d}\sigma_{\uz}\ \mathrm{d}u_0\ \mathrm{d}h\ \mathrm{d}\sigma_{\uu}.
\end{align*}
Performing the integration over $y_0$ formally, a factor $4\pi c\ \delta(x_0-z_0-u_0)$ shows up in the integrand. Rearranging terms, keeping in mind the integration over $\ueta$ and using distributivity, the remaining part of the integrand reads
\begin{align*}
& \phantom{+\ } (1-\uo\uzeta) f(\bfz) \left( e^{-\frac{i}{2c}(r-n-h)g} + e^{-\frac{i}{2c}(r-n+h)g} + e^{\frac{i}{2c}(r-n+h)g} + e^{\frac{i}{2c}(r-n-h)g}\right)\\
& +(1+\uo\uzeta) f(\bfz) \left( e^{-\frac{i}{2c}(r+n-h)g} + e^{-\frac{i}{2c}(r+n+h)g} +e^{\frac{i}{2c}(r+n+h)g} + e^{\frac{i}{2c}(r+n-h)g} \right)\\
& \left.+(\uo\ueta+\ueta\uzeta) f(\bfz)\ueta\unu \left( e^{-\frac{i}{2c}(r-n-h)g} - e^{-\frac{i}{2c}(r-n+h)g} -e^{\frac{i}{2c}(r-n+h)g} + e^{\frac{i}{2c}(r-n-h)g}\right)\right.\\
&\left. +(\uo\ueta-\ueta\uzeta) f(\bfz)\ueta\unu \left( e^{-\frac{i}{2c}(r+n-h)g} - e^{-\frac{i}{2c}(r+n+h)g} -e^{\frac{i}{2c}(r+n+h)g} + e^{\frac{i}{2c}(r+n-h)g} \right)\right.
\end{align*}
times $g(\bfu)$. Using the formal equality
\begin{align*}
\int\displaylimits_{0}^{+\infty} e^{-\frac{i}{2c}(r-n-h)g} \mathrm{d}g = \frac12\int\displaylimits_{0}^{+\infty}e^{-\frac{i}{2c}(r-n-h)g} \mathrm{d}g+\frac12\int\displaylimits_{-\infty}^{0}e^{\frac{i}{2c}(r-n-h)g} \mathrm{d}g
\end{align*}
to dissolve all terms in the previous expression and recombining terms with respect to the integrands, formal delta distributions are obtained in $r$, $n$ and $h$ too. Writing $f,g\in\mathcal{V}=\text{span}_{Cl}\{\psi_{j,k}\}$ as $f(\bfz)=f_1(z_0,n)+\uzeta f_2(z_0,n)$ and $g(\bfu)=g_1(u_0,h)+\unu g_2(u_0,h)$ we get 
\begin{align*}
\mathcal{F}_S^{-1}\left(\mathcal{F}_S(f) \mathcal{F}_S(g)\right)(\bfx)=-i \left(\frac{\Gamma\left(\frac{m}{2}\right)}{8c\pi^{m/2+1}}\right) &\int\displaylimits_{\mathbb{S}^{m-1}} \int\displaylimits_{-\infty}^{+\infty} \int\displaylimits_{0}^{+\infty} \int\displaylimits_{-\infty}^{+\infty} \int\displaylimits_{0}^{+\infty} \delta(x_0-z_0-u_0)\ \times\\
\Big[(\phantom{\ueta}f_1(z_0,n)+\uo \phantom{\ueta} f_2(z_0,n)) \phantom{\ueta} g_1(u_0,h) &\big( \delta(r-n-h) + \delta(r-n+h) \big) \\
 -\uo(\ueta f_1(z_0,n)+\uo\ueta f_2(z_0,n)) \ueta g_2(u_0,h) &\big( \delta(r-n-h) - \delta(r-n+h)\big)\\
 \phantom{\uo}+(\phantom{\ueta}f_1(z_0,n)-\uo \phantom{\ueta} f_2(z_0,n))\phantom{\ueta} g_1(u_0,h) &\big( \delta(r+n-h) + \delta(r+n+h) \big) \\
 -\uo(\ueta f_1(z_0,n)-\uo\ueta f_2(z_0,n)) \ueta g_2(u_0,h) &\big( \delta(r+n-h) - \delta(r+n+h) \big) \Big]\\
&\times \mathrm{d}\sigma_{\uy}\ \mathrm{d}z_0\ \mathrm{d}n\ \mathrm{d}u_0\ \mathrm{d}h,
\end{align*}
where we have integrated over $\uzeta$ and $\unu$. Given that $r,n$ and $h$ can only take positive values, terms containing $\delta(r+n+h)$ do not contribute to the final result.
Using Theorem \ref{ak}, the integration over $\ueta$ yields
\begin{align*}
\mathcal{F}_S^{-1}&\left(\mathcal{F}_S(f) \mathcal{F}_S(g)\right)(\bfx)=\frac{-i}{4\pi c} \int\displaylimits_{-\infty}^{+\infty} \int\displaylimits_{0}^{+\infty} \int\displaylimits_{-\infty}^{+\infty} \int\displaylimits_{0}^{+\infty} \delta(x_0-z_0-u_0)\ \times\\
& \Bigg\{\hspace{0.3ex}\left[f(z_0,n,\phantom{-}\uo)g_1(u_0,h)- \uo \sum_{k=0}^m (-1)^k \left(\frac{2k}{m}-1\right) f^{(k)}(z_0,n,\phantom{-}\uo) g_2(u_0,h) \right]\\
&\quad \times \delta(r-n-h) \\
&+\left[f(z_0,n,\phantom{-}\uo)g_1(u_0,h)+ \uo \sum_{k=0}^m (-1)^k \left(\frac{2k}{m}-1\right)f^{(k)}(z_0,n,\phantom{-}\uo) g_2(u_0,h) \right]\\
& \quad \times \delta(r-n+h) \\
&+\left[f(z_0,n,-\uo)g_1(u_0,h)- \uo \sum_{k=0}^m (-1)^k \left(\frac{2k}{m}-1\right) f^{(k)}(z_0,n,-\uo) g_2(u_0,h) \right]\\
&\quad \times\ \delta(r+n-h) \Bigg\}\ \mathrm{d}z_0\ \mathrm{d}n\ \mathrm{d}u_0\ \mathrm{d}h,
\end{align*}
where the $k$-vector parts of $f_1$ and $f_2$ have been denoted as $f_1^{(k)}$ and $f_2^{(k)}$, respectively. Performing the integrations over $u_0$ and $h$, the expression for the Mustard convolution reads
\begin{align}
&\mathcal{F}_S^{-1}\left(\mathcal{F}_S(f) \mathcal{F}_S(g)\right)(\bfx)=\frac{-i}{4\pi c} \int\displaylimits_{-\infty}^{+\infty} \nonumber \\
& \bigg\{ \hspace{1.7ex}\int\displaylimits_{0}^{r} \left[f_1(z_0,n)g_1(x_0-z_0,r-n)+ \sum_{k=0}^m (-1)^k \left(\frac{2k}{m}-1\right) f_2^{(k)}(z_0,n) g_2(x_0-z_0,r-n) \right]\mathrm{d}n\nonumber\\
&+\int\displaylimits_{r}^{+\infty} \left[f_1(z_0,n)g_1(x_0-z_0,n-r)- \sum_{k=0}^m (-1)^k \left(\frac{2k}{m}-1\right) f_2^{(k)}(z_0,n) g_2(x_0-z_0,n-r) \right]\mathrm{d}n\nonumber\\
&+\int\displaylimits_{0}^{+\infty} \left[f_1(z_0,n)g_1(x_0-z_0,r+n)- \sum_{k=0}^m (-1)^k \left(\frac{2k}{m}-1\right) f_2^{(k)}(z_0,n) g_2(x_0-z_0,r+n) \right]\mathrm{d}n \nonumber\\
&+\uo \int\displaylimits_{0}^{r} \hspace{0.8ex} \left[f_2(z_0,n)g_1(x_0-z_0,r-n)- \sum_{k=0}^m (-1)^k \left(\frac{2k}{m}-1\right) f_1^{(k)}(z_0,n) g_2(x_0-z_0,r-n) \right]\mathrm{d}n\nonumber\\
&+\uo\int\displaylimits_{r}^{+\infty} \left[f_2(z_0,n)g_1(x_0-z_0,n-r)+ \sum_{k=0}^m (-1)^k \left(\frac{2k}{m}-1\right) f_1^{(k)}(z_0,n) g_2(x_0-z_0,n-r) \right]\mathrm{d}n\nonumber\\
&-\uo\int\displaylimits_{0}^{+\infty} \left[f_2(z_0,n)g_1(x_0-z_0,r+n)+ \sum_{k=0}^m (-1)^k \left(\frac{2k}{m}-1\right) f_1^{(k)}(z_0,n) g_2(x_0-z_0,r+n) \right]\mathrm{d}n \bigg\} \nonumber\\
& \mathrm{d}z_0. \label{appendix}
\end{align}

Given the complexity of the resulting expression, we have implemented and thoroughly verified this result using the computersoftware Maple, which confirmed our calculations. 
Nonetheless, a careful study of the different terms suggests some hidden symmetry underneath. Indeed, defining even and odd extensions in the second argument as
\begin{align*}
f^{even}_1(z_0,n)=
\begin{cases}
f_1(z_0,n)\qquad &n>0\\
f_1(z_0,-n)\qquad &n<0
\end{cases}
\intertext{and}
f^{odd}_1(z_0,n)=
\begin{cases}
f_1(z_0,n)\qquad &n>0\\
-f_1(z_0,-n)\qquad &n<0
\end{cases}
\end{align*}
we finally obtain
\begin{align*}
&\mathcal{F}_S^{-1}\left(\mathcal{F}_S(f) \mathcal{F}_S(g)\right)(\bfx)\\
&=\frac{-i}{4\pi c} \int\displaylimits_{-\infty}^{+\infty} \Bigg\{ \int\displaylimits_{-\infty}^{+\infty} \Bigg[f_1^{even}(z_0,n)g_1^{even}(x_0-z_0,r-n)\\
& \phantom{\frac{-i}{4\pi c} \int \uo \bigg\{\uo}\left.+ \sum_{k=0}^m (-1)^k \left(\frac{2k}{m}-1\right) f_2^{odd,(k)}(z_0,n) g_2^{odd}(x_0-z_0,r-n) \right]\mathrm{d}n\\
&\phantom{\frac{-i}{4\pi c} \int \bigg\{}+\uo \int\displaylimits_{-\infty}^{+\infty} \Bigg[f_2^{odd}(z_0,n)g_1^{even}(x_0-z_0,r-n)\\
&\phantom{\frac{-i}{4\pi c} \int \uo \bigg\{\uo} \left.- \sum_{k=0}^m (-1)^k \left(\frac{2k}{m}-1\right) f_1^{even,(k)}(z_0,n) g_2^{odd}(x_0-z_0,r-n) \right]\mathrm{d}n \Bigg\}\mathrm{d}z_0.
\end{align*}
We can thus end this subsection with the following definition and theorem.
\begin{definition}[Mustard convolution] Using $\star$ to denote the classical two-dimensional convolution as stated in \eqref{convolutie}, the Mustard convolution $\star_S$ corresponding to the slice Fourier transform is defined as
\begin{align*}
&f\star_Sg(\bfx)\\
&=\frac{-i}{4\pi c} \Bigg\{ f_1^{even}\star g_1^{even}(x_0,r)+ \sum_{k=0}^m (-1)^k \left(\frac{2k}{m}-1\right) f_2^{odd,(k)}\star g_2^{odd}(x_0,r)\\
&\phantom{=\frac{-i}{4\pi c}}+\uo \left( f_2^{odd}\star g_1^{even}(x_0,r)- \sum_{k=0}^m (-1)^k \left(\frac{2k}{m}-1\right) f_1^{even,(k)}\star g_2^{odd}(x_0,r)\right)\Bigg\},
\end{align*}
where $f,g\in\mathcal{V}$ were written as $f(\bfx)=f_1(x_0,r)+\uo f_2(x_0,r)$ and $g(\bfx)=g_1(x_0,r)+\uo g_2(x_0,r)$ and the superscripts even and odd denote the respective extensions in the second argument.
\end{definition}

\begin{theorem} Under the action of the slice Fourier transform, the Mustard convolution $\star_S$ obeys the classical convolution property
\begin{align*}
\mathcal{F}_S(f\star_S g) = \mathcal{F}_S(f) \mathcal{F}_S(g).
\end{align*}
\end{theorem}

\subsection{The generalised translation}
Another way to generalise the convolution property is by adapting the translation operator $t_y$ in the definition of the classical convolution stated in $\eqref{convolutie}$. Using $T_{\bfy}$ to denote this generalised translation operator, the Fubini theorem requires
\begin{align*}
\cF_S\left(\ \int\displaylimits_{\mR^{m+1}}T_{\bfy} (f)(\bfx) g(\bfy) \mathrm{d}\bfy \right) (\bfz)&= \int\displaylimits_{\mR^{m+1}} \mathcal{F}_S (T_{\bfy} (f))(\bfz)\ g(\bfy)\ \mathrm{d}\bfy
\end{align*}
to be equal to
\begin{align*}
\left(\mathcal{F}_S(f) \mathcal{F}_S(g)\right)(\bfz)&=\int\displaylimits_{\mR^{m+1}}\mathcal{F}_S (f)(\bfz)\ \mathcal{K}^M(\bfy,\bfz)\ g(\bfy)\ \mathrm{d}\bfy,
\end{align*}
which is achieved when the integrands are equal. The generalised translation operator $T_{\bfy}$ is thus defined as
\begin{align}
T_{\bfy} f(\bfx) = \mathcal{F}_S^{-1}\left( \mathcal{F}_S(f)(\bfz) K^M(\bfy,\bfz)\right),\label{gentrans}
\end{align}
so its slice Fourier transform generates an extra factor $\cK^M(\bfy,\bfz)$ to the right of the slice Fourier transform of the function it is acting upon.
\begin{remark}
Note the crucial difference between the above defined generalised translation $T_{\bfy}$ and the translation operator studied in \cite{de2011clifford}. As the slice Fourier transform of $T_{\bfy}(f)$ adds the factor $\cK^M$ to the right of $\cF_S(f)$, the kernel expression $\cK^M$ can subsequently be right-multiplied with a neighbouring factor.
\end{remark}
Denoting the kernel function of the inverse slice Fourier transform as $\cK^{M,-1}$, the following expression is shown to equal the inverse Fourier transform of the product of the slice Fourier transforms of $f$ and $g$:
\begin{align*}
\int\displaylimits_{\mR^{m+1}}&T_{\bfy}(f)(\bfx)g(\bfy)\mathrm{d}\bfy\\
&= \int\displaylimits_{\mR^{m+1}}\left(\ \int\displaylimits_{\mR^{m+1}}  \cK^{M,-1}(\bfz,\bfx) \mathcal{F}_S(f)(\bfz) \cK^M(\bfy,\bfz)g(\bfy) \ \mathrm{d}z_0\mathrm{d}n\mathrm{d}\sigma_{\uz}\right)\mathrm{d}y_0\mathrm{d}g\mathrm{d}\sigma_{\uy}\\
&= \int\displaylimits_{\mR^{m+1}}  \cK^{M,-1}(\bfz,\bfx) \mathcal{F}_S(f)(\bfz) \left(\ \int\displaylimits_{\mR^{m+1}} \cK^M(\bfy,\bfz)\ g(\bfy) \ \mathrm{d}y_0\mathrm{d}g\mathrm{d}\sigma_{\uy}\right) \mathrm{d}z_0\mathrm{d}n\mathrm{d}\sigma_{\uz}\\
&= \int\displaylimits_{\mR^{m+1}}  \cK^{M,-1}(\bfz,\bfx) \left(\mathcal{F}_S(f)(\bfz) \mathcal{F}_S(g)(\bfz)\right) \ \mathrm{d}z_0\mathrm{d}n\mathrm{d}\sigma_{\uz}\\
&= \mathcal{F}_S^{-1} \left(\mathcal{F}_S(f) \mathcal{F}_S(g)\right)(\bfx).
\end{align*}
As the above defined generalised translation operator $T_{\bfy}$ shows the desired behaviour, we will now calculate its explicit expression. Substituting all known kernel functions and writing $\bfx=x_0+r\sigma_{\ux}, \bfy=y_0+g\sigma_{\ueta}, \bfz=z_0+h\sigma_{\uzeta}$ and $\bfu=u_0+n\sigma_{\unu}$ yields
\begin{align*}
T_{\bfy}&(f)(\bfx) = -i \left( \frac{\Gamma\left(\frac{m}{2}\right)}{8c\pi^{m/2+1}} \right)^3 \int\displaylimits_{\mR^{m+1}} \int\displaylimits_{\mR^{m+1}} \\
& \left[ (1+\uo\uzeta) e^{+\frac{i}{2c}(x_0z_0-rn)} + (1-\uo\uzeta) e^{+\frac{i}{2c}(x_0z_0+rn)} \right] \times \\
&\left[ (1+\uzeta\unu) e^{-\frac{i}{2c}(z_0u_0-nh)} + (1-\uzeta\unu) e^{-\frac{i}{2c}(z_0u_0+nh)} \right] f(\bfu)\ \times \\
& \left[ (1+\uzeta\ueta) e^{-\frac{i}{2c}(y_0z_0-gn)}\hspace{0.3ex} + (1-\uzeta\ueta)\hspace{0.3ex} e^{-\frac{i}{2c}(y_0z_0+gn)} \right] \mathrm{d}z_0\ \mathrm{d}n\ \mathrm{d}\sigma_{\uz}\ \mathrm{d}u_0\ \mathrm{d}h\ \mathrm{d}\sigma_{\uu}.
\end{align*}
As usual, integrating over $z_0$ formally gives a delta distribution. Rearranging the remaining terms we get
\begin{align*}
T_{\bfy} &(f)(\bfx) = -i \left( \frac{\Gamma\left(\frac{m}{2}\right)}{8c\pi^{m/2+1}} \right)^3 \int\displaylimits_{0}^{+\infty} \int\displaylimits_{\mathbb{S}^{m-1}}  \int\displaylimits_{\mR^{m+1}} 4\pi c\ \delta(x_0-y_0-u_0)\ \times \\
\bigg\{&\left[ (1+\uo\uzeta +\uzeta\unu-\uo\unu) e^{-\frac{i}{2c}(r-h-g)n} + (1+\uo\uzeta -\uzeta\unu+\uo\unu) e^{-\frac{i}{2c}(r+h-g)n} \right. \\
&\phantom{.}\left. (1-\uo\uzeta +\uzeta\unu+\uo\unu)\hspace{0.3ex} e^{+\frac{i}{2c}(r+h+g)n} + (1-\uo\uzeta -\uzeta\unu-\uo\unu)\hspace{0.3ex} e^{+\frac{i}{2c}(r-h+g)n} \right] \\
&\phantom{.}\times f(\bfu)(1+\uzeta\ueta)\\
+ &\left[ (1+\uo\uzeta +\uzeta\unu-\uo\unu) e^{-\frac{i}{2c}(r-h+g)n} + (1+\uo\uzeta -\uzeta\unu+\uo\unu) e^{-\frac{i}{2c}(r+h+g)n} \right. \\
&\phantom{.}\left. (1-\uo\uzeta +\uzeta\unu+\uo\unu) e^{+\frac{i}{2c}(r+h-g)n} + (1-\uo\uzeta -\uzeta\unu-\uo\unu) e^{+\frac{i}{2c}(r-h-g)n} \right] \\
&\phantom{.} \times f(\bfu)(1-\uzeta\ueta) \bigg\}\ \mathrm{d}n\ \mathrm{d}\sigma_{\uz}\ \mathrm{d}u_0\ \mathrm{d}h\ \mathrm{d}\sigma_{\uu}.
\end{align*}
Because of symmetry reasons, all terms containing $\uzeta$ will disappear after integration over $\mS^{m-1}$. As before, only functions in the span of $\{\psi_{j,k}\}$ are considered so $f$ can be written as $f(\bfu)=f_1(u_0,h)+\unu f_2(u_0,h)$. Performing the same manipulations as before (see section \ref{sectie2}) on the integrals over $[0,+\infty[$ and integrating over $n, \uzeta$ and $\unu$ yields
\begin{align*}
T_{\bfy} (f)(\bfx) =& \frac{-i\Gamma\left(\frac{m}{2}\right)}{8c\pi^{m/2+1}} \int\displaylimits_{-\infty}^{+\infty} \int\displaylimits_{0}^{+\infty} \delta(x_0-y_0-u_0) \times\\
\bigg\{&\left[f(u_0,h,\phantom{-}\uo)+\uo\sum_{k=0}^m (-1)^k\left(\frac{2k}{m}-1\right) \left(f_1^{(k)}+\uo f_2^{(k)}\right)(u_0,h) \ueta \right]\\
&\times\ \delta(r-h-g)\\
+&\left[f(u_0,h,\phantom{-}\uo)-\uo\sum_{k=0}^m (-1)^k\left(\frac{2k}{m}-1\right) \left(f_1^{(k)}+\uo f_2^{(k)}\right)(u_0,h) \ueta \right]\\
&\times\ \delta(r-h+g)\\
+&\left[f(u_0,h,-\uo)+\uo\sum_{k=0}^m (-1)^k\left(\frac{2k}{m}-1\right) \left(f_1^{(k)}-\uo f_2^{(k)}\right)(u_0,h) \ueta \right]\\
&\times\ \delta(r+h-g) \bigg\}\ \mathrm{d}u_0\ \mathrm{d}h.
\end{align*}
Performing the last two integrations, we finally obtain the explicit action of the generalised translation $T_{\bfy}$ on a function $f$:
\begin{align*}
T_{\bfy} &(f)(\bfx) =\frac{-i\Gamma\left(\frac{m}{2}\right)}{8c\pi^{m/2+1}}\ \times\\
&\bigg\{\hspace{0.9ex}f(x_0-y_0,r-g,\phantom{-}\uo)+\uo\sum_{k=0}^m (-1)^k\left(\frac{2k}{m}-1\right) f^{(k)}(x_0-y_0,r-g,\phantom{-}\uo) \ueta \\
&+f(x_0-y_0,r+g,\phantom{-}\uo)-\uo\sum_{k=0}^m (-1)^k\left(\frac{2k}{m}-1\right) f^{(k)}(x_0-y_0,r+g,\phantom{-}\uo) \ueta \\
&+f(x_0-y_0,g-r,-\uo)+\uo\sum_{k=0}^m (-1)^k\left(\frac{2k}{m}-1\right) f^{(k)}(x_0-y_0,g-r,-\uo) \ueta \bigg\}
\end{align*}
where terms in $f$ and $f^{(k)}(\bfx)=f_1^{(k)}(x_0,r)+\uo f_2^{(k)}(x_0,r)$ are only included when their second arguments are positive. Making the $\uo$-dependence explicit yields
\begin{align*} 
T_{\bfy} f(\bfx) &=\frac{-i\Gamma\left(\frac{m}{2}\right)}{8c\pi^{m/2+1}}\ \times \\
\bigg\{&\left[f_1(x_0-y_0,r-g)-\sum_{k=0}^m (-1)^k\left(\frac{2k}{m}-1\right) f_2^{(k)}(x_0-y_0,r-g) \ueta \right]\\
+&\left[f_1(x_0-y_0,r+g)+\sum_{k=0}^m (-1)^k\left(\frac{2k}{m}-1\right) f_2^{(k)}(x_0-y_0,r+g) \ueta \right]\\
+&\left[f_1(x_0-y_0,g-r)+\sum_{k=0}^m (-1)^k\left(\frac{2k}{m}-1\right) f_2^{(k)}(x_0-y_0,g-r) \ueta \right]\\
+&\uo\left[ \phantom{-}f_2(x_0-y_0,r-g)+\sum_{k=0}^m (-1)^k\left(\frac{2k}{m}-1\right) f_1^{(k)}(x_0-y_0,r-g) \ueta \right]\\
+&\uo\left[\phantom{-}f_2(x_0-y_0,r+g)-\sum_{k=0}^m (-1)^k\left(\frac{2k}{m}-1\right) f_1^{(k)}(x_0-y_0,r+g) \ueta \right]\\
+&\uo\left[- f_2(x_0-y_0,g-r)+\sum_{k=0}^m (-1)^k\left(\frac{2k}{m}-1\right) f_1^{(k)}(x_0-y_0,g-r) \ueta \right]\bigg\}.
\end{align*}
where again terms in $f_1$ and $f_2$ are only included when their second argument is positive (so depending on $r \in [0,+\infty[$ being smaller or bigger than $g\in[0,+\infty[$).

\subsection{Connection between both approaches}

\begin{proposition}
The Mustard convolution $\star_S$ and the generalised translation $T_{\bfy}$ satisfy
\begin{align}\label{equal}
(f\star_S g) (\bfx)= \int_{\mR^{m+1}}T_{\bfy}(f)(\bfx)g(\bfy)\ \mathrm{d}\bfy.
\end{align}
\end{proposition}
\begin{proof}
Taking the slice Fourier transform of both sides of \eqref{equal} gives the same expression: $\cF(f)\cF(g)$. In section \ref{explicit} it was shown how the slice Fourier transform (and therefore also its inverse) can be calculated using the classical two-dimensional Fourier transform. Given that the latter is a bijective transform, the above equality follows immediately.
\end{proof}

We can also verify this proposition explicitly:
\begin{align*}
\int\displaylimits_{\mR^{m+1}}&T_{\bfy} (f)(\bfx) g(\bfy) \mathrm{d}\bfy =\frac{-i\Gamma\left(\frac{m}{2}\right)}{8c\pi^{m/2+1}} \int\displaylimits_{-\infty}^{+\infty} \int\displaylimits_{0}^{+\infty} \int\displaylimits_{\mathbb{S}^{m-1}}\\
&\bigg\{\hspace{0.9ex}f(x_0-y_0,r-g,\phantom{-}\uo)+\uo\sum_{k=0}^m (-1)^k\left(\frac{2k}{m}-1\right) f^{(k)}(x_0-y_0,r-g,\phantom{-}\uo) \ueta \\
&+f(x_0-y_0,r+g,\phantom{-}\uo)-\uo\sum_{k=0}^m (-1)^k\left(\frac{2k}{m}-1\right) f^{(k)}(x_0-y_0,r+g,\phantom{-}\uo) \ueta \\
&+f(x_0-y_0,g-r,-\uo)+\uo\sum_{k=0}^m (-1)^k\left(\frac{2k}{m}-1\right) f^{(k)}(x_0-y_0,g-r,-\uo) \ueta \bigg\}\\
&\times g(\bfy)\ \mathrm{d} y_0\ \mathrm{d} g\ \mathrm{d}\sigma_{\uy}
\end{align*}
where the different terms are only taken into account when the second argument of $f$ or $f^{(k)}$ is positive. Writing $g(\bfy)=g_1(y_0,g)+\ueta g_2(y_0,g)$ and performing the integration over $\ueta$, we obtain

\begin{align*}
&\int\displaylimits_{\mR^{m+1}}T_{\bfy} (f)(\bfx) g(\bfy) \mathrm{d}\bfy =\frac{-i}{4\pi c} \int\displaylimits_{-\infty}^{+\infty} \int\displaylimits_{0}^{+\infty} \int\displaylimits_{\mathbb{S}^{m-1}}\\
&\bigg\{\hspace{0.9ex}f(x_0-y_0,r-g,\phantom{-}\uo)g_1(y_0,g)-\uo\sum_{k=0}^m (-1)^k\left(\frac{2k}{m}-1\right) f^{(k)}(x_0-y_0,r-g,\phantom{-}\uo) g_2(y_0,g) \\
&+f(x_0-y_0,r+g,\phantom{-}\uo)g_1(y_0,g)+\uo\sum_{k=0}^m (-1)^k\left(\frac{2k}{m}-1\right) f^{(k)}(x_0-y_0,r+g,\phantom{-}\uo) g_2(y_0,g) \\
&+f(x_0-y_0,g-r,-\uo)g_1(y_0,g)-\uo\sum_{k=0}^m (-1)^k\left(\frac{2k}{m}-1\right) f^{(k)}(x_0-y_0,g-r,-\uo) g_2(y_0,g) \bigg\}\\
& \mathrm{d} y_0\ \mathrm{d} g.
\end{align*}
Performing the substitutions $z_0=x_0-y_0$ and $n=\pm r\pm g$, where the signs depend on the integrand, we indeed retrieve the expression for $\mathcal{F}_S^{-1}(\mathcal{F}_S(f)\mathcal{F}_S(g))$ from \eqref{appendix}.

\section{Conclusion}
The article at hand introduces the classical Fourier transform to the slice Clifford setting. Using the Clifford-Hermite functions defined in \cite{Cnudde} as a basis of eigenfunctions, corresponding eigenvalues are obtained from the scalar differential equation they satisfy. The combination of these data allowed for a formal definition of the slice Fourier transform. Its explicit expression was obtained using the Mehler and Hille-Hardy formulas. The closed integral expression evoked a closer study of its main characteristics including its differential and symmetry properties and the construction of the inverse transform. Also the computational load of the slice Fourier transform is briefly addressed.\\
The second part of the paper investigates how to construct an appropriate counterpart of the classical convolution. For its slice Fourier transform to equal the product of the slice Fourier transforms of the functions it is acting upon, two possibilities are presented. First the Mustard convolution is studied and its explicit expression is obtained. In a second approach, a generalised translation operator is introduced to the integrand of the convolution expression. Given that both approaches show the same extension of the classical convolution property, the paper finishes by highlighting their connection.

\section{Acknowledgements}
L. Cnudde and H. De Bie are supported by the UGent BOF starting grant 01N01513.

\end{document}